\documentclass{article}

\usepackage{arxiv}
\usepackage{epsfig,amsmath,amssymb,amsthm}
\usepackage{ctex}           
\usepackage[utf8]{inputenc} 
\usepackage[T1]{fontenc}    
\usepackage{hyperref}       
\usepackage{url}            
\usepackage{booktabs}       
\usepackage{amsfonts}       
\usepackage{nicefrac}       
\usepackage{microtype}      
\usepackage{cleveref}       
\usepackage{lipsum}         
\usepackage{graphicx}
\usepackage{doi}
\usepackage{extarrows}
\usepackage{color}          
\newtheorem{Theorem}{Theorem}[section]
\newtheorem{Corollary}{Corollary}[section]
\newtheorem{Definition}{Definition}[section]

\newtheorem{Lemma}{Lemma}[section]
\newtheorem{Proposition}{Proposition}[section]

\newtheorem{Remark}{Remark}[section]

\numberwithin{equation}{section}

\ctexset{
	bibname = {References},
}

\usepackage{cite}
\usepackage{enumerate}   
\begin{document}
	\title{ Some study of the approximate orthogonality connected with integral orthogonalities }
	\author{Ranran Wang, Qi Liu\thanks{ Corresponding author. Qi Liu,   E-mail: liuq67@aqnu.edu.cn} , Jinyu Xia$ $ ,Yongmo Hu }
 
	\date{}
	\maketitle
	
	\begin{center}
		{\small School of Mathematics and Physics,   Anqing Normal University,   Anqing 246133,   P.    R.    China}
	\end{center}

%
%
%

\begin{abstract}
\noindent
         \quad \quad In this paper, we investigate a novel form of approximate orthogonality that is based on integral orthogonality. Additionally, we establish the fundamental properties of this new approximate orthogonality and examine its capability to preserve mappings of orthogonality. Moreover, we explore the relationship between this new approximate orthogonality and other forms of approximate orthogonality.
         
	{\bf{Keywords}\rm} {Approximate orthogonality; Inner produce spaces; Approximate orthogonality preserving mapping}\\
		{\small\bf MR(2020) Subject Classification\ \ {\rm46B20, 46C05}}
\end{abstract}
 	
	\section{Introduction}
	\setlength{\parindent}{2em}

   Let $X$ be inner product spaces and  $x, y \in X$. it's said that $x$ is orthogonality to $y$ if and only if $\langle x \mid y\rangle=0$ (called $x \perp y$ ). The concept of orthogonality is introduced widely in inner product spaces, such as Birkhoff orthogonality\cite{4}, Isosceles orthogonality\cite{2,4}, $\delta $ orthogonality  etc. Then, we recall the following

    \begin{enumerate}[\rm(i)]
    \item  Birkhoff orthogonality $ x \perp_ {B} y$ if $\|x+\alpha  y\| \geq\|x\|$ for all $\alpha \in \mathbb{R}$;
    
    \item  Isosceles orthogonality $x \perp _I y$ if $\|x+y\|=\|x-y\| $;
    \item $\delta $ orthogonality $ x \perp^{\delta} y $ if $ \mid\langle x \mid y\rangle \mid \leq \delta \|x\| \|y\| $. 
   \end{enumerate}

        Obviously, all orthogonalities are similar to $x\perp y$. Thus, there is a natural way to generalize orthogonality is to define a new orthogonality (we called it approximate orthogonality) by $x \perp^{\varepsilon} y$ if and only if
       \begin{align*}
      |\langle x \mid y\rangle| \leq \varepsilon\|x\|\| y \|,
      \end{align*}
      for all $x, y \in X$. Dragomir \cite{R2,5,6} defined the approximate Birkhoff \cite{13,R3} orthogonality $x^{\varepsilon} \perp_B y$, if and only if 
      \begin{align*}
      \|x+t y\| \geq(1-\varepsilon)\|x\|
      \end{align*}
       for all $t \in \mathbb{R}$. Apparently, we can find that this type of approximate orthogonality is same to $\perp^{\varepsilon}$ in inner product spaces. Then, Chmielin\'ski \cite{R1,7} found the approximate Birkhoff orthogonality \cite{12,13} $x \perp^{\varepsilon}_B y$ if and only if 
              \begin{align*}
       \|x+t y\|^2 \geq\|x\|^2-2 \varepsilon\|x\| \|t y\|
          \end{align*}
        for all $t \in \mathbb{R}$.
        After that, the approximate isosceles orthogonality \cite{R2}  $x \perp^{\varepsilon}_I y$ if and only if
        \begin{align*}
        	 \left|\|x+y\|^2-\|x-y\|^2\right| \leq 4 \varepsilon\|x\| \|y\|
        \end{align*} for all $t \in \mathbb{R}$ and $x^{\varepsilon}\perp_I y$ if and only if  \begin{align*}
        |\|x+y\|-\|x-y\||<\varepsilon\|x+y\| \|x-y\|    
        \end{align*}
         for all $t \in \mathbb{R}$ were introduced.
         
         In inner product spaces\cite{8,10}, we can show the property of orthogonality about linear mapping. Let $X$ and $Y$ be inner product spaces with an orthogonal relation, and $f: X \rightarrow Y$ which satisfies that if $x \perp y$, and then $f(x) \perp f(y)$ for all $x , y \in X$ (called orthogonality preserving). Similarly, approximately  orthogonality preserving\cite{11} is introduced. For any $\varepsilon \in[0,1)$ and let $x, y \in X$, $g: X \rightarrow Y$ which satisfies if $x \perp y$, and then $f(x) \perp^{\varepsilon} f(y)$.
         
      In the study, we will introduce a new type of orthogonality called approximate $H H-I$  orthogonality.This orthogonality has two different variations, which we will explore in depth. We will investigate the relationship between this new orthogonality and other approximate orthogonalities, and analyze its basic properties. Additionally, we give the approximate similarities with  linear approximate $H H-I$ orthogonality preserving mappings.

	\section{Approximate $HH-I$ orthogonality}
Now, we will start the section by two essential definitions and some key propositions of approximate $H H-I$ orthogonality. Next, we introduce the definition of $HH-I$ orthogonality of Hermite-Hadamand orthogonality (see \cite{R4,R5,R6}) given by Silvestru Sever Dragomir.

\begin{Definition}
	Let $\varepsilon \in[0,1)$ and any $x, y \in X$, a vector $x$ is said to be $HH-I$ orthogonal to $y$ if they satisfy
\end{Definition}
\begin{align*}
   \int_0^1\|(1-t) x+t y\|^2 d t=\int_0^1\|(1-t) x-t y\|^2 d t
\end{align*}
  for all $t \in \mathbb{R}$, denoted by $x \perp_{H H-I} y$.
  
  Similarly, we can  define  various forms of approximate $H H-I$ orthogonality based on $HH-I$ orthogonality, which is called $\varepsilon-H H-I$ orthogonality.
  
  \begin{Definition}
  		Let $\varepsilon \in[0,1)$ and any $x, y \in X$, a vector $x$ is said to be $\varepsilon-HH-I$ orthogonal to $y$ if they satisfy
  	\begin{align*}
  		&\left| \int_0^1\|(1-t) x+t y\|^2 d t-\int_0^1\|(1-t) x-t y\|^2 d t \right| \\
  		\leq& \varepsilon\left(\int_0^1\|(1-t) x+t y\|^2 d t+\int_0^1\|(1-t) x-t y\|^2 d t\right),
  	\end{align*}
  	for all $t \in \mathbb{R}$, denoted by $x^{\varepsilon} \perp_{H H-I} y  $.
  \end{Definition}
 The above inequality can be simplified to
  \begin{align*}
  	&-\varepsilon\left(\int_0^1\|(1-t) x+t y\|^2 d t+\int_0^1\|(1-t) x-t y\|^2 d t\right)\\
  	 \leq& \int_0^1\|(1-t) x+t y\|^2 d t-\int_0^1\|(1-t) x-t y\|^2 d t 
  	\\
  	\leq& \varepsilon\left(\int_0^1\|(1-t) x+t y\|^2 d t+\int_0^1\|(1-t) x-t y\|^2 d t\right).
   \end{align*}
  	  \quad \quad It is easy to find that it is equivalent to 
  	\begin{align*} 
  	&\frac{1-\varepsilon}{1+\varepsilon} \int_0^1\|(1-t) x-t y\|^2 d t \\
  	\leq& \int_0^1\|(1-t) x+t y\|^2 d t \\
  	\leq& \frac{1+\varepsilon}{1-\varepsilon} \int_0^1\|(1-t) x-t y\|^2 d t.
\end{align*}
  
As we all know, the structure of the Hilbert space is well, and inspired by this, we propose the other definition of $\varepsilon-H H-I$ orthogonality as follows.

  \begin{Definition}
  	Let $\varepsilon \in[0,1)$ and any $x, y \in X$, a vector $x$ is said to be $\varepsilon-HH-I$ orthogonal to $y$ if they satisfy
  \begin{align*}
  	\left|\int_0^1\|(1-t) x+t y\|^2 d t-\int_0^1\|(1-t) x-t y\|^2 d t\right| \leq \frac{2}{3} \varepsilon\|x\|\|y\|,
  \end{align*}
  for all $t \in \mathbb{R}$, denoted by $x \perp_{H H-I}^{\varepsilon} y $.
  \end{Definition}

  It is obvious that above approximate $H H-I$ orthogonality is same to $H H-I$ orthogonality for $\varepsilon=0$. Simply, we can observe that the second definition is stronger than the first, that is, if $x \perp_{H H-I}^{\varepsilon} y$, then $x^{\varepsilon} \perp_{H H-I} y$.
   However, the reverse is not by the next example. 
   
   Let $Y$ be a real valued inner product and all $x, y \in X$ , it is also easy to check the following
   \begin{align}\label{Pro3}
   x \perp_{H H-I}^{\varepsilon} y \Leftrightarrow |\langle x \mid y\rangle| \leq \varepsilon\|x\|\|y\| \Leftrightarrow x \perp^{\varepsilon} y,
     \end{align}
   and
   \begin{align*}
   x^{\varepsilon} \perp_{H H-I} y \Leftrightarrow|\langle x \mid y\rangle|\leq \frac{\varepsilon}{1+\varepsilon^2}\left(\|x\|^2+\|y\|^2\right) .
   \end{align*}
   
  Then, the first approximate $HH-I$-orthogonality is same to the standard approximate orthogonalities in the inner product space.
   
  Next, we show some basic properties of $\varepsilon-HH-I $ orthogonality. 
   
      \begin{Proposition} Let any $\varepsilon \in[0,1)$, and  all  $x,y \in X$. Then the relations $\perp_{H H-I}^{\varepsilon}$ and ${ }^{\varepsilon} \perp_{H H-I}$ ane symmetric. Therefore,
   	\begin{enumerate}[\rm(i)]
   		\item if $x \perp_{H H-I}^{\varepsilon} y$, then $y \perp_{H H-I}^{\varepsilon} x $;
   		\item if $x^{\varepsilon} \perp_{H H-I} y$, then $y^{\varepsilon} \perp_{H H-I} x $.
   	\end{enumerate}
   \end{Proposition}
   
   \begin{proof}
   	\begin{enumerate}[\rm(i)]
   	Now, we will prove the symmetry of approximate $H H-I$ orthogonality with two definitions.
   	
   	(i) Let $x \perp^{\varepsilon}_{H H-I} y$, Then
   	\begin{align*}
   		\left|\int_0^1\|(1-t) y+t x\|^2 d t-\int_0^1\|(1-t) y-t x\|^2 d t\right| \leq
   		\frac{2}{3} \varepsilon\|y\|\|x\|
   	\end{align*}
   	for any $t \in \mathbb{R} \backslash\{0\}$. Thus $y \perp_{H H-I}^{\varepsilon} x$.
   	
   	(ii) Let $x ^ {\varepsilon}\perp_{H H-I} y$, Then
   	\begin{align*}
   			& \left|\int_0^1\|(1-t) y+t x\|^2 d t-\int_0^1\|(1-t) y-t x\|^2 d t\right| 
\\ 			 
   			  \leq& \varepsilon\left(\int_0^1\|(1-t) y+t x\|^2 d t+\int_0^1\|(1-t) y-t x\|^2 d t\right) \
   			\end{align*}
   	for any $t \in \mathbb{R} \backslash\{0\}$. Thus $y^{\varepsilon} \perp_{H H-I} x$.
   		\end{enumerate}
   \end{proof}
   \begin{Proposition} Suppose that $\varepsilon \in[0,1)$	 and $x, y \in X$, Then the relations $\perp^{\varepsilon}{ }_{H H-I}$ and ${ }^{\varepsilon} \perp_{H H-I}$ are homogeneous. Therefore, for any $\alpha$, $\beta \in \mathbb{R}$, have
   	\begin{enumerate}[\rm(i)]
   		\item if $x \perp_{H H-I}^{\varepsilon} y$, then $ \alpha x \perp_{H H-I}^{\varepsilon} \beta y $;
   		\item if $x^{\varepsilon} \perp_{H H-I} y$, then $  \alpha x^{\varepsilon} \perp_{H H-I} \beta y$.
   	\end{enumerate}
   \end{Proposition}
   \begin{proof} 	Now, we will prove the homogeneity of approximate $H H-I$ orthogonality with two definitions.	
   	
   	 (i) For any $t \in \mathbb{R}$, due to $x \perp^{\varepsilon}_{ HH-I} y$, We have (if $\alpha=0$, the next inequality is obviously true)
   	\begin{align*}
   		\left|\int_0^1\|(1-t) \alpha x+t \beta y\|^2 d t-\int_0^1\|(1-t) \alpha x-t \beta y\|^2 d t\right| 
   		\leq \frac{2}{3} \varepsilon\|\alpha x\|\| \beta y\| .
   	\end{align*}
   	Thus, $\alpha x \perp_{H H-I}^{\varepsilon} \beta y$.
   	
   	(ii) For any $t \in \mathbb{R}$, due to $x^{\varepsilon} \perp_{ HH-I} y$, We have (if $\alpha=0$, the next inequality is obviously true)
   	\begin{align*}
   	\begin{aligned}
   		 &\left|\int_0^1\|(1-t) \alpha x+t \beta y\|^2 d t-\int_0^1\|(1-t) \alpha x-t \beta y\|^2 d t\right| \\
   		\leq &\varepsilon\left(\int_0^1\|(1-t) \alpha x+t \beta y\|^2 d t+\int_0^1\|(1-t)\alpha x-t \beta y\|^2 d t\right).
   	\end{aligned}   	
   	\end{align*}	
   	Thus, $\alpha x^{\varepsilon} \perp_{H H-I}\beta y$.
   \end{proof}
  \section{Approximately $HH-I$ orthogonality preserving mappings}

Next, we will define approximate $HH-I$ orthogonality preserving mapping. Let $X$ be an inner product space and  $x, y \in X$, if $x$ is orthogonal to $y$, then their images $f(x)$ is approximate $HH-I$ orthogonality to $f(y)$ (with ${}^{\varepsilon}\perp_{H H-I}$ or $\perp^{\varepsilon}_{H H-I}$ ). In the part, we will show that linear approximate $H H-I$ orthogonality preserving mappings are approximately similar.

  We consider a linear and continuous mapping $f: X \rightarrow Y$ and  define norm of $f$  as following
 \begin{align*}
   \|f\|=\sup \{\|f x\|:\|x\|=1\}=\inf \{A>0:\|f x\| \leq A\|x\|, x \in X\} .
 \end{align*}
   Similarly, we define
  \begin{align*}
   [f]:=\inf \{\|f x\|:\|x\|=1\}=\sup \{A \geq 0:\|f x\| \geq A\|x\|, x \in X\} .
  \end{align*}
  To get our next result, we need  the following lemma.
  \begin{Lemma}\label{Lemma}\cite{12} 
  	Assume that $0 \leq \alpha \leq 1 \leq \beta, 0 \leq C \leq D$, a linear mapping $g: X \rightarrow Y$ satisfies
  \begin{align}\label{L1}
  \alpha D\|x\|^2 \leq\|g x\|^2 \leq \beta C\|x\|^2, \quad x \in X
\end{align}\label{L2}
  if and only if it satisfies
  \begin{align}
  \alpha \eta\|x\|^2 \leq\|g x\|^2 \leq \beta \eta\|x\|^2
\end{align}
for all $x \in X$, $\eta \in[C, D]$.
  \end{Lemma}
  
   \begin{Theorem}\label{Th1} Let $\varepsilon \in[0,1)$ and for all $ x, y \in X $. Set $g: X \rightarrow Y$ be a nontrivial linear mapping satisfying
   \begin{align}\label{11}
   x \perp_{H H-I} y \Longrightarrow g(x)^{\varepsilon} \perp_{H H-I} g(y).
     \end{align}
   Then, $g$ is injective, continuous and satisfies
   \begin{align}\label{12}
   \frac{1-\varepsilon}{1+\varepsilon}\|g\|^2\|x\|^2 \leq\|g x\|^2 \leq \frac{1+\varepsilon}{1-\varepsilon}[g]^2\|x\|^2, \quad x \in X,
     \end{align}
   or equivalently
\begin{align}\label{13}
   \frac{1-\varepsilon}{1+\varepsilon} \eta^2\|x\|^2 \leq\|g x\|^2 \leq \frac{1+\varepsilon}{1-\varepsilon} \eta^2\|x\|^2, \quad x \in X, \quad \eta \in[[g],\|g\|] .
\end{align}
   Conversely, if  $g: X \rightarrow Y$ is a linear bounded mapping and satisfies \eqref{12} (or \eqref{13}), then it satisfies \eqref{11}.
  \end{Theorem}
   \begin{proof}
   	 Taking $\frac{x+y}{1-t}$ and $\frac{x-y}{t}$ instead  of $x$ and $y$ in \eqref{11}, we can get the following equivalent form
   \begin{align}\label{14}
   \|x\|=\|y\| \Rightarrow|\|g x\|^2-\|g y\|^2|<\varepsilon(\|g x\|^2+\|g y\|^2), 
    \end{align}
   for all $ x, y \in X $. Set any vector $y \in X$ and $\|y\|=1$  and let $\eta:=\|g y\|$. By\eqref{14}, we can get
   \begin{align*}
    \|x\|^2=1  \Rightarrow |\|g x\|^2-\eta^2| \leq \varepsilon(\|g x\|^2+\eta^2),
    \end{align*}
  for all  $x \in X $. Whence
   \begin{align*}
   \left|\left\|g\left(\frac{x}{\|x\|}\right)\right\|^2-\eta^2\right| \leq \varepsilon\left(\left\|g\left(\frac{x}{\|x\|}\right)\right\|^2+\eta^2\right), \quad x \in X \backslash\{0\},
    \end{align*}
   and
 \begin{align*}
   |\|g x\|^2-(\eta\|x\|)^2|<\varepsilon(\|g x\|^2+(\eta\|x\|)^2), \quad x \in X .
   \end{align*}
   This is equivalent to
   \begin{align*}
   \frac{1-\varepsilon}{1+\varepsilon} (\eta\|x\|)^2 \leq\|g x\|^2 \leq \frac{1+\varepsilon}{1-\varepsilon} (\eta\|x\|)^2, \quad x \in X .
    \end{align*}
   Thus, we conclude that $g$ is injective and continuous.
 
   Since we set if as $\|fy\|$ for $\|y\|=1$ and $y \in X$. We get \eqref{12} following from Lemma \ref{Lemma} through the supremum and intimum, then we can get that \eqref{13} holds true with 
   \begin{align*}
   	\alpha=\frac{1-\varepsilon}{1+\varepsilon}, \beta=\frac{1+\varepsilon}{1-\varepsilon}, C=[g]^2, D=\|g\|^2.
   \end{align*}
 
   To prove the reverse statement, we suppose \eqref{12} holds and let $u, w \in X$ be such that $u \perp_{H H-I} w$, ie.,
  \begin{align*}
   \int_0^1\|(1-t) u+t w\|^2 d t=\int_0^1\|(1-t) u-t w\|^2 d t.
    \end{align*}
   Assume $u \neq w$ (otherwise $u=w=0$ and the assertion holds trivially). We define
    \begin{align*}
   \eta_0:=\left\|\left(\frac{\int_0^1\|g((1-t) u+t w)\|^2 d t}{\int_0^1\|(1-t) u+t w\|^2 d t}\right)\right\|^{\frac{1}{2}} \in[[g ],\|g \|] .
\end{align*}
   From \eqref{13}, we have
    \begin{align*}
   \frac{1-\varepsilon}{1+\varepsilon} \eta_0^2\|x\|^2 \leq\|g x\|^2 \leq \frac{1+\varepsilon}{1-\varepsilon} \eta_0^2\|x\|^2,
\end{align*}
  for any $x \in X$, which is equivalent to
   \begin{align*}
   \left|\int_0^1\|g x\|^2 d t-\eta_0^2 \int_0^1\|x\|^2 d t\right| \leq \varepsilon\left(\int_0^1\|g x\|^2 d t+\eta_0^2 \int_0^1\|x\|^2 d t\right), x \in X .
\end{align*}
   Then, we put $(1-t) u+t w$ instead of $x$ , then
   \begin{align*}
   	&\bigg|\int_0^1\|g((1-t) u+t w)\|^2 d t-\eta_0^2 \int_0^1\|(1-t) u+t w\|^2 d t\bigg| \\
   	 \leq& \varepsilon\left(\int_0^1\|g((1-t) u+t w)\|^2 d t+\eta_0^2 \int_0^1\|(1-t) u+t w\|^2 d t\right) .
\end{align*}
   Since $g$ is linear and  $ \int_0^1\|(1-t) u+t w\|^2 d t=\int_0^1\|(1-t) u-t w\|^2 d t$ , we obtain
   \begin{align*}
   	&\left|\int_0^1\|(1-t) g u+t g w\|^2 d t-\int_0^1\|(1-t) g u-t g w\|^2 d t\right| \\
   	=&\left|\int_0^1\|g((1-t) u+t w)\|^2 d t-\eta_0^2 \int_0^1\|(1-t) u+t w\|^2 d t\right| \\
   	 \leq& \varepsilon\left(\int_0^1\|g((1-t) u+t w)\|^2 d t+\eta_0^2 \int_0^1\|(1-t) u+t w\|^2 d t\right) \\
   	=& \varepsilon \left(\int_0^1\|(1-t) g u+t g w\|^2 d t+\int_0^1\|(1-t) g u-t g w\|^2 d t\right),
   \end{align*}
  by the definition of $\eta_0$. Thus, we can get $gu^{\varepsilon} \perp_{HH-I}gw$. Consequently, we prove the desired result.
\end{proof}
   \begin{Remark}
   Clearly, for any $x, y \in X$, we can assert that for a linear mapping $g: X \rightarrow Y$ satisfying a stronger condition than \eqref{11} as following
    \begin{align*}
   x \perp_{H H-I} y \Rightarrow g(x) \perp_{H H-I}^{\varepsilon} g(y).
\end{align*}
   For $\varepsilon=0$ and $x, y \in X$ (ie., for $H H-I$ orthogonality preserving mappings) and by \eqref{12},we have
     \begin{align*}
   \|g\|^2\|x\|^2 \leq\|g x\|^2 \leq[g]^2\|x\|^2 \leq\|g\|^2\|x\|^2.
\end{align*}
   Whence $\|g x\|^2=\eta^2\|x\|^2, $ for any $x \in Y$ with $\eta^2=\|g\|^2=[g]^2$.
   \end{Remark}
   Next, set $\|\cdot\|_1$ and $\|\cdot\|_2$ be two norms in $Y$ and let $\perp_{H H-I_1}, \perp_{H H-I_2}$ denote the $H H-I$ orthogonality relations with respect to the first or the second norm. According to Theorem \ref{Th1}, we know that if $\perp_{H H-I_1} \subseteq{ }^{\varepsilon} \perp_{HH-I_2}$, that is, if $x \perp_{H H-I_1} y$, then $x^{\varepsilon} \perp_{H H-I_2} y$, for all $x, y \in Y$, hence for any $\eta$, we have
    \begin{align*}
   \inf _{\|x\|_1=1}\|x\|_2 \leq \eta \leq \sup _{\|x\|_1=1}\|x\|_2.
\end{align*}
   Therefore, for any $ x\in X$
   \begin{align*}
   \frac{1-\varepsilon}{1+\varepsilon} (\eta\|x\|_1)^2 \leq\|x\|_2^2 \leq \frac{1+\varepsilon}{1-\varepsilon} (\eta\|x\|_1)^2 .
\end{align*}
   Set $\varepsilon=0$, since $\perp_{H H-I_1} \subseteq \perp_{H H-I_2}$, then for any $ x \in X, \exists  \eta>0$
   \begin{align*}
   \|x\|_2=\eta\|x\|_1.
\end{align*}
   
   \begin{Corollary}
   	 Let  $Y$ be a real vector space and $\|\cdot\|_1$ and $\|\cdot\|_2$ be two equivalent norms in $Y$. Thus, \rm(i), \rm(ii) and \rm(iii) are equivalent.
   	 	\begin{enumerate}[\rm(i)]
   \item 
   $\perp_{H H-I_1} \subseteq \perp_{H H-I_2}$;
   \item 
   $\perp_{H H-I_1}=\perp_{H H-I_2}$;
   \item 
   $\|x\|_2=\eta\|x\|_1, x \in Y$, where  some $\eta>0$ .
   \end{enumerate}
   \end{Corollary}
   \begin{Corollary}
   	 Let $g: X \rightarrow Y$ be a nontrivial linear mapping. Thus, $g$ satisfies \eqref{11} if and only if $g$ satisfies
    \begin{align}\label{16}
   \|g\|^2 \leq \frac{1+\varepsilon}{1-\varepsilon}[g]^2 .
\end{align}
    \end{Corollary}
   \begin{proof} 
   By  Theorem \ref{Th1}, we can get
   \begin{align*}
   \|gx\|^2 \leq \frac{1+\varepsilon}{1-\varepsilon}[g]^2 \|x\|^2.
   \end{align*}
   Let $\|x\|=1$, we have
   \begin{align*}
   \|g\|^2 \leq \frac{1+\varepsilon}{1-\varepsilon}[g]^2.
   \end{align*}
   \end{proof}
   \begin{Corollary}
   	 Let $g: X \rightarrow Y$ be a nontrivial linear mapping. Then $g$ satisfies \eqref{11} if and only if $g$ satisfies
   \begin{align}\label{17}
   \|g x\|^2\|y\|^2 \leq \frac{1+\varepsilon}{1-\varepsilon}\|g y\|^2\|x\|^2, \quad x, y \in X.
\end{align}
   \end{Corollary}
   \begin{proof}
   	 If $g$ satisfies \eqref{11}. At first, for $y=0$, \eqref{17} holds trivially.
   	  Then, we assume that $y \neq 0$ and set 
   \begin{align*}
  	\|gx\|^2 \leq \frac{1+\varepsilon}{1-\varepsilon}\|gy\|^2,\|x\|=\|y\|=1.
  	 \end{align*}
  Thus
  \begin{align*}
  \|gx\|^2 \|y\|^2 \leq \frac{1+\varepsilon}{1-\varepsilon}\|gy\|^2 \|x\|^2. 
   \end{align*}
  	Conversely, we assume \eqref{17} holds. 
	Passing to the supremum over $\|x\|=1$. 
Then, we get
     \begin{align*} 
    \|g\|^2 \|y\|^2 \leq \frac{1+\varepsilon}{1-\varepsilon}\|gy\|^2. 
 \end{align*}
  	Set $\|y\|=1$, we have
  	 \begin{align*}
  	 \|g\|^2 \leq \frac{1+\varepsilon}{1-\varepsilon}\|gy\|^2.
  	 \end{align*}
  		And passing to the infimum over $\|y\|=1$, we obtain
  	\begin{align*}
  	 \|g\|^2 \leq \frac{1+\varepsilon}{1- \varepsilon}[g]^2,
    \end{align*}
  	  which is equivalent to  
  	  \begin{align*}
  	  	\|gx\|^2 \leq \frac{1+\varepsilon}{1-\varepsilon}\|gy\|^2,\|x\|=\|y\|=1.
  	  \end{align*}
 	Consequently, we get a desired result.	 
   \end{proof}
 
 \begin{Proposition}
 Let $Y$ be a real normed space, and	suppose that there are two same norms on it, ie., $m\|x\|_1 \leq\|x\|_2 \leq M\|x\|_1$ for all $x \in Y$ and some $0<m \leq M$. If $	x \perp_{H H-I, 1} y$, then $x^{\eta} \perp_{H H-I, 2} y$
 	for all $x, y \in Y$, where $\eta=\frac{M-m}{M+m}$ and $\perp_{{H H-I}, i}$ denotes the $H H-I$ orthogonality with respect to the norm $\|\cdot\|_i$.
 \end{Proposition}
 \begin{proof}
 	Assume $x \perp_{{H H-I},1} y$, we get
 	\begin{align*}
 		\frac{m}{M} \int_0^1\|(1-t) x-t y\|_2^2 d t & \leq m \int_0^1\|(1-t) x+t y\|_1^2 d t \\
 		& \leq \int_0^1\|(1-t) x+t y\|_2^2 d t \\
 		& \leq M \int_0^1\|(1-t) x-t y\|_1^2 d t \\
 		& \leq \frac{M}{m} \int_0^1\|(1-t) x-t y\|_2^2 d t .
 	\end{align*}
 	Due to $\frac{1+\eta}{1-\eta}=\frac{M}{m}$, we get
 	\begin{align*}
 		&\frac{1-\eta}{1+\eta} \int_0^1\|(1-t) x-t y\|_2^2 d t \\ 
 		\leq& \int_0^1\|(1-t) x+t y\|_2^2 d t \\
 		\leq& \frac{1+\eta}{1-\eta} \int_0^1 \|(1-t) x-t y\|_2^2 d t.
 	\end{align*}
 	Therefore, $x^\eta \perp_{H H-I,2} y$.
 \end{proof}
 
Based on reference \cite{13}, we propose improvements to the following lemma.
 \begin{Lemma}\label{Lemma1}
	If $Y$ be a real normed space and let $x, y \in Y$. Then
	\begin{align*}
		\min \left\{\left\|\frac{x}{\beta}\right\|^2+\|\beta y\|^2: \beta\neq 0\right\}=2\|x\|\|y\| .
	\end{align*}
\end{Lemma}
\begin{proof}
	Consider the following inequality 
	\begin{align*}
		\|\frac{x}{\beta}\|^2+\|\beta y\|^2\geq& 2 \cdot \big|\frac{1}{\beta}\big|\|x\|    \cdot\big| \beta \big| \|y \|\\
	\geq &2\|x\|\|y\|,
	\end{align*}
	if and only if $\displaystyle\|\frac{x}{\beta}\|=\|\beta y\|$, the equality $\displaystyle \left\|\frac{x}{\beta}\right\|^2+\|\beta y\|^2=2\|x\|\|y\| $ holds.
\end{proof}
\begin{Theorem}\label{Th2}
	 If  $Y$ is a real-valued inner product space, let $\varepsilon \in[0,1)$ and $x, y \in Y$, then
	\begin{align*}
		x^{\varepsilon} \perp_{H H-I} y \Leftrightarrow x \perp^\delta y,
	\end{align*}
	where $\delta={2 \varepsilon}$.
\end{Theorem} 
\begin{proof}
	Let $x^{\varepsilon} \perp_{H H-I} y$ and $t \in \mathbb{R}$. Then
	\begin{align*}
		\frac{1-\varepsilon}{1+\varepsilon}\int_0^1\|(1-t) x-t y\|^2 d t \leq \int_0^1\|(1-t) x+t y\|^2 d t.
	\end{align*}
	Therefore
	\begin{align*}
		& \frac{1-\varepsilon}{1+\varepsilon} \int_0^1\left(\left(1-t)^2\right\| x\left\|^2+\right\| t y \|^2-2\langle(1-t) x \mid t y\rangle\right) d t \\
		\leq & \int_0^1\left(\left\|(1-t)^2\right\| x\left\|^2+\right\| t y \|^2+2\langle( 1-t) x|t y\rangle\right) d t,
	\end{align*}
	which has the following equivalent form
	\begin{align*}
		&\frac{-2\varepsilon}{1+\varepsilon}\left(\int_0^1(1-t)^2\|x\|^2+t^2\|y\|^2 d t\right)\\ \leq&\frac{2}{1+\varepsilon} \int_0^1 2 t(1-t)\langle x \mid y\rangle d t .
	\end{align*}
	Thus
	\begin{align*}
		{\frac{-2\varepsilon}{1+\varepsilon}\left(\|x\|^2+\|y\|^2\right) \leq\frac{2}{1+\varepsilon}\langle x \mid y\rangle,} 
	\end{align*}
	and
	\begin{align*}
		{-\varepsilon}\left(\|x\|^2+\|y\|^2\right) \leq\langle x \mid y\rangle .
	\end{align*}
	Similarly,
	\begin{align*}
		\langle x \mid y\rangle \leq {\varepsilon}\left(\|x\|^2+\|y\|^2\right) .
	\end{align*}
	Hence,
	\begin{align*}
		|\langle x \mid y\rangle| \leq {\varepsilon}\left(\|x\|^2+\|y\|^2\right) .
	\end{align*}
	Lemma \ref{Lemma1} now leads to
	\begin{align*}
		|\langle x \mid y\rangle| &\leq {\varepsilon} \min \left\{\left\|\frac{x}{\beta}\right\|^2+\|\beta y\|^2: \beta>0\right\}\\
		&={2 \varepsilon}\|x\|\|y\| ,
	\end{align*}
	which show that $x \perp^\delta y$.
	
	 Otherwise , if $x \perp^\delta y$, then $|\langle x \mid y\rangle| \leq \delta\|x\|\|y\|$. By Lemma \ref{Lemma1}, let {$\beta=1$}, we get
	\begin{align*}
		\delta\|x\|\|y\| \leq {\varepsilon}\left(\|x\|^2+\|y\|^2\right),
	\end{align*}
	for any $t \in \mathbb{R}$. Similar to the inverse process of the above proof, we have
	\begin{align*}
		\frac{1-\varepsilon}{1+\varepsilon}\int_0^1\|(1-t) x-t y\|^2 d t \leq \int_0^1\|(1-t) x+t y\|^2 d t.
	\end{align*}
	Moreover, we have the following
	\begin{align*}
		&\frac{1-\varepsilon}{1+\varepsilon}\int_0^1\|(1-t) x-t y\|^2 d t\\
		\leq& \int_0^1\|(1-t) x+t y\|^2 d t \\
		\leq&\frac{1+\varepsilon}{1-\varepsilon} \int_0^1\|(1-t) x-t y\|^2 d t .
	\end{align*}
	Which gives $x^{\varepsilon} \perp_{H H-I} y$, we completed the proof about this theorem.
\end{proof}
\begin{Corollary}  For  a real-valued inner product space $Y$, let $\varepsilon \in[0,1)$ and all $x, y \in Y$, then
	\begin{align*}
		x \perp_{H H-I}^{\varepsilon} y \Leftrightarrow x^\eta \perp_{H H-I} y,
  	\end{align*}
	where $\eta=\frac{1-\sqrt{1-\varepsilon^2}}{\varepsilon}$.
\end{Corollary}
\begin{proof} Due to $\eta=\frac{1-\sqrt{1-\varepsilon^2}}{\varepsilon}$, then we get $\varepsilon=\frac{2 \eta}{1+\eta^2}$. By Theorem \ref{Th2} for $\varepsilon=\frac{2 \eta}{1+\eta^2}$,thus
	\begin{align*}
		x \perp^{\varepsilon} y \Leftrightarrow x^\eta \perp_{H H-I} y .
	\end{align*}
 \eqref{Pro3} now leads to
	\begin{align*}
		x \perp_{H H-I}^{\varepsilon} y \Leftrightarrow x^\eta \perp_{H H-I} y .
	\end{align*}
\end{proof}

 \section*{Acknowledgments}
	 Thanks to all the members of the Functional Analysis Research team of the College of Mathematics and Physics of Anqing Normal University for their discussion and correction of the difficulties and errors encountered in this paper.  

	\end{document}